\documentclass[twoside]{article}
\usepackage{graphics}
\usepackage{graphicx}
\usepackage{latexsym}
\usepackage{amssymb,amsfonts,amsmath}
\usepackage{amsthm} % this defines the proof environment
\usepackage{BSPMsample}
\usepackage{enumitem}
\numberwithin{equation}{section}
\usepackage{multirow}%

\usepackage{multicol}

\newtheorem{theorem}{Theorem}%  meant for continuous numbers
\newtheorem{proposition}[theorem]{Proposition}% 

\newtheorem{example}{Example}%
\newtheorem{remark}{Remark}%
\newtheorem{definition}{Definition}%
\newtheorem{lemma}[theorem]{Lemma}
\begin{document}

\title[ Algebraic and Graphical Analysis of H-Toeplitz operators on Fock space]{ Algebraic and Graphical Analysis of H-Toeplitz operators on Fock space}
%between [] goes the title that will appear on the top of every
%odd page, between {} goes the real title.
%\thanks{...} put your grant

\author[Th. Sonamani Singh, M. P Singh, O. Nilbir Singh and Kh. Priyobarta Singh]{Thokchom Sonamani Singh, Moirangthem Premjit Singh, Oinam Nilbir Singh and Khumballambam Priyobarta Singh}
%between [] goes the authors that will appear on the top of every
%even page, between {} goes the real author.

\address{Prof. M. Premjit Singh,\\Department of Mathematics, \\
Manipur University, Imphal \\ India}
\email{mpremjitmu@gmail.com}

\maketitle

\begin{abstract}
This paper presents a comprehensive study of H-Toeplitz operators on the Fock space, a class of operators that synthesizes structural elements of both Toeplitz and Hankel operators. We derive explicit matrix representations for these operators with respect to the standard orthonormal basis of monomials, providing a foundational tool for their analysis. Central to our investigation are the algebraic and spectral properties of these operators. We establish precise conditions for commutativity, particularly for operators with harmonic symbols, and prove that non-zero H-Toeplitz operators cannot be Hilbert-Schmidt. Furthermore, we develop a Mellin transform-based framework to characterize the hyponormality and normality of operators with quasi-homogeneous symbols, deriving verifiable analytical criteria. Finally, we introduce the novel concept of directed H-Toeplitz graphs to visualize the adjacency relations encoded by the matrix structure of these operators. This graphical representation reveals distinct and interpretable patterns in indegree and outdegree sequences, offering a new combinatorial lens through which to understand their structure. Our results forge a significant connection between the analytic theory of operators on function spaces and the discrete structures of graph theory, enriching both fields.
\end{abstract}
%%put here Key words


\tableofcontents

%% put here the subject class
\2010mathclass{47B35, 32A25.}

\section{Introduction}\label{sec1}
Let $\alpha > 0$ be a fixed parameter, and consider the Gaussian measure $d\lambda_\alpha$ on $\mathbb{C}$ defined by  
\[
d\lambda_\alpha(z) = \frac{\alpha}{\pi} e^{-\alpha |z|^2} \, dA(z),  
\]  
where $dA$ denotes the Lebesgue area measure. For $0 < p < \infty$, the Fock space $F_\alpha^p$ consists of all entire functions $f$ such that  
\[
\|f\|_p = \left( \frac{p\alpha}{2\pi} \int_{\mathbb{C}} |f(z)|^p e^{-\frac{p\alpha}{2}|z|^2} \, dA(z) \right)^{1/p} < \infty.  
\]  
When $p = 2$, $F_\alpha^2$ is a reproducing kernel Hilbert space with inner product  
\[
\langle f, g \rangle = \int_{\mathbb{C}} f(z) \overline{g(z)} \, d\lambda_\alpha(z),  
\]  
and reproducing kernel $K_z(w) = e^{\alpha \overline{z} w}$. The orthogonal projection $P: L^2(\mathbb{C}, d\lambda_\alpha) \to F_\alpha^2$ is given by  
\[
Pf(z) = \int_{\mathbb{C}} f(w) e^{\alpha \overline{z} w} \, d\lambda_\alpha(w).  
\]  

Toeplitz and Hankel operators on $F_\alpha^2$ are defined in analogy with their classical counterparts on Hardy and Bergman spaces. Given a symbol $\varphi \in L^\infty(\mathbb{C})$, the Toeplitz operator $T_\varphi$ and Hankel operator $H_\varphi$ are given by  
\[
T_\varphi f = P(\varphi f), \quad H_\varphi f = PM_\varphi J(f).  
\]

The study of Toeplitz and Hankel operators has been a central theme in operator theory and functional analysis, with significant developments occurring across various function spaces. In particular, the Hardy and Bergman spaces have served as natural settings for investigating these operators and their generalizations. The introduction of H-Toeplitz operators by Arora and Paliwal \cite{ar} marked an important advancement, unifying aspects of both Toeplitz and Hankel operators while exhibiting novel structural properties.
On the Hardy space, Arora and Paliwal established a comprehensive framework for H-Toeplitz operators, characterizing their partial isometric, co-isometric, Hilbert-Schmidt, and hyponormal properties. Notably, they demonstrated that every H-Toeplitz operator is unitarily equivalent to a direct sum of a Toeplitz operator and a Hankel operator, revealing a deep connection between these classes. This structural insight underscores the significance of H-Toeplitz operators, as they bridge the gap between Toeplitz and Hankel operators while introducing new analytical phenomena.

The investigation of H-Toeplitz operators was later extended to the Bergman space, where their properties were systematically explored. Gupta and Singh \cite{gu} initiated this study, examining fundamental aspects such as compactness, Fredholmness, and commutativity. 
Hyponormality, a key subnormality condition where $[T^*,T] \geq 0$, emerged as a focal point. Gupta and Aggarwal \cite{gupta2024} showed that sums of Toeplitz operators with non-harmonic monomial symbols are hyponormal under specific conditions, while Aggarwal and Gupta \cite{aggarwal2025} extended this to quasi-homogeneous symbols, proving rigidity for positive-degree cases and polynomial coefficients.Subsequent work by Kim and Lee \cite{kim} provided criteria for contractivity and expansivity, further enriching the understanding of these operators. Additionally, Liang et al. \cite{liang} explored commutativity conditions for H-Toeplitz operators with quasi-homogeneous symbols, while Ding and Chen \cite{ding} characterized the conditions under which products of H-Toeplitz operators retain the H-Toeplitz structure.
Parallel developments in the study of Toeplitz and Hankel operators on the Fock space \cite{cho, str, xu, ful,hu} have contributed to a broader understanding of operator theory in analytic function spaces.

In this paper, we develop the theory of H-Toeplitz operators on the Fock space $F_{\alpha}^{2}$. While the matrix representations of standard Toeplitz and Hankel operators on Fock spaces are known, a detailed analysis of H-Toeplitz operators in this setting, including their explicit matrix forms, algebraic properties, and a groundbreaking graphical interpretation remains unexplored. We bridge this gap by providing a complete framework for their analysis. Our work goes beyond prior results in several key aspects: we establish conditions for commutativity; prove the non-existence of non-zero Hilbert-Schmidt H-Toeplitz operators; and develop a novel Mellin transform approach to fully characterize the hyponormality and normality of operators with quasi-homogeneous symbols, providing explicit, verifiable criteria. Furthermore, we introduce a new graphical model by defining directed H-Toeplitz graphs whose structure directly reflects the analytic and anti-analytic components of the operator's symbol, creating a powerful visual and combinatorial tool for understanding these operators. This synthesis of operator algebra, spectral theory, and graph theory provides profound new insights into the structure of H-Toeplitz operators on the Fock space.

\section{Preliminaries}\label{sec2}
\begin{definition}%[Multiplication Operator]
Let \(\phi\) be a measurable function. Define the multiplication operator \(M_\phi\) on \(L^2(\mathbb{C}, d\lambda_\alpha)\) by
\[
M_\phi(f) = \phi f \quad \text{for all } f \in L^2(\mathbb{C}, d\lambda_\alpha).
\]
\end{definition}

\begin{definition}%[Toeplitz Operator]
\cite{josh}
Given a bounded measurable function \(\phi \in L^\infty(\mathbb{C})\), the Toeplitz operator \(T_\phi\) on \(F^2_\alpha\) with symbol \(\phi\) is defined by
\[
T_\phi(f) = P(\phi f), \quad \text{for all } f \in F^2_\alpha.
\]
It is easy to verify that \(T_\phi\) is a bounded linear operator on \(F^2_\alpha\), satisfying the norm estimate
\[
\lVert T_\phi f \rVert \leq \lVert \phi \rVert_\infty \lVert f \rVert, \quad \text{for all } f \in F^2_\alpha,
\]
which implies \(\lVert T_\phi \rVert \leq \lVert \phi \rVert_\infty\).
\end{definition}

\begin{definition}%[Hankel Operator]
For a bounded measurable symbol \(\phi \in L^\infty(\mathbb{C})\), the Hankel operator \(H_\phi\) on \(F^2_\alpha\) is defined by
\[
H_\phi(f) = P M_\phi J(f), \quad \text{for all } f \in F^2_\alpha,
\]
where \(J : F^2_\alpha \to F^2_\alpha\) is the so-called \emph{flip operator}, acting on the orthonormal basis \(\{e_n\}\) by
\[
J(e_n) = \overline{e_{n+1}}, \quad \text{for } n = 0, 1, 2, \ldots
\]
\end{definition} 
%  \section{H-Toeplitz operators on Fock space $F_\alpha^2$} 
%We begin this section with some lemmas which are used in subsequent results.
%Recall the functions $e_n(z)=\sqrt{\frac{\alpha^n}{n!}} z^n$, $n=0, 1, 2, 3,$ ... form an orthonormal basis for $F_\alpha^2$.\\

\begin{lemma}\label{lem:monomial-orthogonality}
In the Fock space $F_\alpha^2$, the monomials $\{z^n\}_{n=0}^\infty$ form an orthogonal system with
\[
\langle z^s, z^t \rangle = \delta_{st} \frac{s!}{\alpha^s},
\]
where $\delta_{st}$ denotes the Kronecker delta.
\end{lemma}

\begin{proof}
Using polar coordinates $z = re^{i\theta}$ and the Gaussian measure $d\mu(z) = \frac{\alpha}{\pi} e^{-\alpha |z|^2} dA(z)$, we compute:
\begin{align*}
\langle z^s, z^t \rangle &= \frac{\alpha}{\pi} \int_{\mathbb{C}} z^s \overline{z}^t e^{-\alpha |z|^2} dA(z) \\
&= \frac{\alpha}{\pi} \int_0^{2\pi} e^{i(s-t)\theta} d\theta \int_0^\infty r^{s+t+1} e^{-\alpha r^2} dr.
\end{align*}

The angular integral vanishes when $s \neq t$:
\[
\int_0^{2\pi} e^{i(s-t)\theta} d\theta = 2\pi \delta_{st}.
\]

For the radial integral when $s = t$, we make the substitution $u = \alpha r^2$:
\begin{align*}
\int_0^\infty r^{2s+1} e^{-\alpha r^2} dr &= \frac{1}{2\alpha^{s+1}} \int_0^\infty u^s e^{-u} du \\
&= \frac{\Gamma(s+1)}{2\alpha^{s+1}} = \frac{s!}{2\alpha^{s+1}}.
\end{align*}

Combining these results yields
\[
\langle z^s, z^t \rangle = \delta_{st} \cdot \frac{\alpha}{\pi} \cdot 2\pi \cdot \frac{s!}{2\alpha^{s+1}} = \delta_{st} \frac{s!}{\alpha^s}. 
\]
\end{proof}
\begin{lemma}\label{lem:projection-monomials}
Let $k, \ell \in \mathbb{N}_0$ be non-negative integers. The orthogonal projection $P$ in $F_\alpha^2$ satisfies:
\begin{enumerate}[label=(\alph*)]
    \item For $k \geq \ell$, the projection is given by
    \[
    P(z^{n+k}\overline{z}^\ell) = \frac{(n+k)!}{\alpha^{n+k} }\sqrt{\frac{\alpha^{n+k-\ell}}{(n+k-\ell)!}} \, e_{n+k-\ell}.
    \]
    
    \item For $\ell > k$, the projection vanishes when $n < \ell - k$ and equals
    \[
     P(z^{n+k}\overline{z}^\ell) = \frac{(n+k)!}{\alpha^{n+k} }\sqrt{\frac{\alpha^{n+k-\ell}}{(n+k-\ell)!}} \, e_{n+k-\ell},
    \]
    when $n \geq \ell - k$.
\end{enumerate}
\end{lemma}

\begin{proof}
We analyze both cases using the orthonormal basis $e_m(z) = \sqrt{\frac{\alpha^m}{m!}} z^m$ of $F_\alpha^2$.

\noindent\textbf{Case (a): $k \geq \ell$} \\
For any $m \geq 0$, the inner product decomposes as:
\begin{align*}
\langle z^{n+k}\overline{z}^\ell, e_m \rangle 
&= \sqrt{\frac{\alpha^m}{m!}} \langle z^{n+k}, z^{m+\ell} \rangle_{F_\alpha^2} \\
&= \sqrt{\frac{\alpha^m}{m!}} \cdot \frac{(n+k)!}{\alpha^{n+k}} \delta_{n+k,m+\ell}.
\end{align*}
By Lemma~\ref{lem:monomial-orthogonality}, the projection is nonzero only when $m = n+k-\ell \geq 0$, yielding:
\begin{align*}
P(z^{n+k}\overline{z}^\ell) &= \sqrt{\frac{\alpha^{n+k-\ell}}{(n+k-\ell)!}} \frac{(n+k)!}{\alpha^{n+k}} e_{n+k-\ell}. \\
\end{align*}

\noindent\textbf{Case (b): $\ell > k$} \\
The same calculation applies, but now
\begin{itemize}
    \item[(i)] For $0 \leq n < \ell - k$, the condition $m = n+k-\ell < 0$ is impossible, so the projection vanishes.
    \item[(ii)] For $n \geq \ell - k$, we recover the same expression as in Case (a) since $m = n+k-\ell \geq 0$.
\end{itemize}
\end{proof}
\section{Operator Matrices on Fock Space}\label{sec3}

We now establish detailed matrix representations for fundamental operators on the Fock space $F_\alpha^2$ with respect to the orthonormal basis $\left\{e_n(z) = \sqrt{\frac{\alpha^n}{n!}}z^n\right\}_{n=0}^\infty$.

\subsection{Toeplitz Operators}

\begin{theorem}%[Matrix of Toeplitz Operator]
For a bounded harmonic symbol $\phi(z) = \displaystyle\sum_{i=0}^\infty a_i z^i + \sum_{j=1}^\infty b_j \overline{z}^j \in L^\infty(\mathbb{C})$, the matrix elements of $T_\phi$ are given by
\[
\langle T_\phi e_n, e_m \rangle = 
\begin{cases}
\sqrt{\frac{\alpha^{n-m} m!}{n!}} \, a_{m-n} & \text{for } m \geq n \\
\sqrt{\frac{\alpha^{m-n} n!}{m!}} \, b_{n-m} & \text{for } n > m \, 
\end{cases}
\]
for non-negative integers $m$ and $n$.
\end{theorem}

\begin{proof}
Here,
\begin{align*}
\langle T_\phi \, e_n, e_m \rangle&= \langle P_\phi \, e_n, e_m \rangle \cr
&=\sqrt{\frac{\alpha^n}{n!}} \sqrt{\frac{\alpha^m}{m!}} \left(\sum\limits_{i=0}^\infty a_i \langle z^{i+n}, z^m \rangle+\sum\limits_{j=1}^\infty b_j \langle z^n, z^{m+j} \rangle \right).
\end{align*}
The matrix elements decompose into two cases: \\
Case 1.  For $m \geq n$, only the analytic part contributes:
\begin{align*}
\langle T_\phi \, e_n, e_m \rangle&=\sqrt{\frac{\alpha^n}{n!}} \sqrt{\frac{\alpha^m}{m!}} \left(\sum\limits_{i=0}^\infty a_i \langle z^{i+n}, z^m \rangle \right)\cr
&=\sqrt{\frac{\alpha^n}{\alpha^m}} \times \sqrt{\frac{m!}{n!}} \,\,\, a_{m-n}\cr&=\sqrt{\frac{\alpha^{n-m} m!}{n!}} \, a_{m-n} .
\end{align*}
Case 2.  For $n > m$, only the anti-analytic part contributes:
\begin{align*}
\langle T_\phi \, e_n, e_m \rangle&=\sqrt{\frac{\alpha^n}{n!}} \sqrt{\frac{\alpha^m}{m!}} \left(\sum\limits_{j=1}^\infty b_j \langle z^n, z^{m+j} \rangle \right)\cr
&=\sqrt{\frac{\alpha^m}{\alpha^n}} \times \sqrt{\frac{n!}{m!}} \,\,\, b_{n-m}\cr&=\sqrt{\frac{\alpha^{m-n} n!}{m!}} \, b_{n-m}
\end{align*}
for non-negative integers $m$ and $n$.
 Its matrix with respect to the orthonormal basis is given by
\[
T_\phi=
\begin{bmatrix}
a_0 & \frac{1}{\sqrt{\alpha}}b_1 & \frac{\sqrt{2}}{\sqrt{\alpha^2}}b_2 & \frac{\sqrt{6}}{\sqrt{\alpha^3}}b_3 & \hdots \\
\frac{1}{\sqrt{\alpha}}a_1 & a_0 & \frac{\sqrt{2}}{\sqrt{\alpha}}b_1 & \frac{\sqrt{6}}{\sqrt{\alpha^2}}b_2 & \hdots \\
\frac{\sqrt{2}}{\sqrt{\alpha^2}}a_2 & \frac{\sqrt{2}}{\sqrt{\alpha}}a_1 & a_0 & \frac{\sqrt{3}}{\sqrt{\alpha}}b_1 & \hdots \\
\frac{\sqrt{6}}{\sqrt{\alpha^3}}a_3 & \frac{\sqrt{6}}{\sqrt{\alpha^2}}a_2 & \frac{\sqrt{3}}{\sqrt{\alpha}}a_1 & a_0 & \hdots \\
\vdots & \vdots & \vdots & \vdots & \ddots
\end{bmatrix}
\]
and the adjoining of the matrix of $T_\phi$ is given by 
\[
T_\phi^*=
\begin{bmatrix}
\overline{a_0} & \frac{1}{\sqrt{\alpha}}\overline{a_1} & \frac{\sqrt{2}}{\sqrt{\alpha^2}}\overline{a_2} & \frac{\sqrt{6}}{\sqrt{\alpha^3}}\overline{a_3} & \hdots \\
\frac{1}{\sqrt{\alpha}}\overline{b_1} & \overline{a_0} & \frac{\sqrt{2}}{\sqrt{\alpha}}\overline{a_1} & \frac{\sqrt{6}}{\sqrt{\alpha^2}}\overline{a_2} & \hdots \\
\frac{\sqrt{2}}{\sqrt{\alpha^2}}\overline{b_2} & \frac{\sqrt{2}}{\sqrt{\alpha}}\overline{b_1} & \overline{a_0} & \frac{\sqrt{3}}{\sqrt{\alpha}}\overline{a_1} & \hdots \\
\frac{\sqrt{6}}{\sqrt{\alpha^3}}\overline{b_3} & \frac{\sqrt{6}}{\sqrt{\alpha^2}}\overline{b_2} & \frac{\sqrt{3}}{\sqrt{\alpha}}\overline{b_1} & \overline{a_0} & \hdots \\
\vdots & \vdots & \vdots & \vdots & \ddots
\end{bmatrix}
\]
which is nothing but the matrix of $T_{\overline{\phi}}$ and therefore $T_\phi^*=T_{\overline{\phi}}$ \,.
\end{proof}
\subsection{Hankel Operators}

\begin{theorem}%[Matrix of Hankel Operator]
The Hankel operator $H_\phi$ has matrix elements
\[
\langle H_\phi e_n, e_m \rangle = \frac{(m+n+1)!}{\sqrt{\alpha^{m+n+1} m! (n+1)!}} a_{m+n+1}\, 
\]
for non-negative integers $m$ and $n$.
\end{theorem}

\begin{proof}
Using the anti-analytic projection
\begin{align*}
 \langle H_\phi e_n, e_m \rangle &= \langle P M_\phi J e_n, e_m \rangle \\
&=\sqrt{\frac{\alpha^{n+m+1}}{(n+1)! m!}} \left(\sum^\infty_{i=0} a_i \langle z^i, z^{m+n+1} \rangle+\sum^\infty_{j=1} \langle \overline{z}^j, z^{m+n+1} \rangle\right)\\
&= \frac{(m+n+1)!}{\sqrt{\alpha^{m+n+1} m! (n+1)!}} a_{m+n+1}   
\end{align*}
for non-negative integers $m$ and $n$.
\end{proof}
The Hankel matrix exhibits a distinct pattern
\[
H_\phi=
\begin{bmatrix}
\frac{1}{\sqrt{\alpha}}a_1 &  \frac{\sqrt{2}}{\sqrt{\alpha^2}}a_2 & \frac{\sqrt{6}}{\sqrt{\alpha^3}}a_3 & \frac{2\sqrt{6}}{\sqrt{\alpha^4}}a_4 & \hdots \\
\frac{2}{\sqrt{\alpha^2}}a_2 & \frac{3\sqrt{2}}{\sqrt{\alpha^3}}a_3 & \frac{4\sqrt{6}}{\sqrt{\alpha^4}}a_4 & \frac{10\sqrt{6}}{\sqrt{\alpha^5}}a_5 & \hdots \\
\frac{3\sqrt{2}}{\sqrt{\alpha^3}}a_3 & \frac{12}{\sqrt{\alpha^4}}a_4 & \frac{20\sqrt{3}}{\sqrt{\alpha^5}}a_5 & \frac{60\sqrt{3}}{\sqrt{\alpha^6}}a_6 & \hdots \\
\frac{4\sqrt{6}}{\sqrt{\alpha^4}}a_4 & \frac{20\sqrt{3}}{\sqrt{\alpha^5}}a_5 & \frac{120}{\sqrt{\alpha^6}}a_6 & \frac{210\sqrt{4}}{\sqrt{\alpha^7}}a_7 & \hdots \\
\vdots & \vdots & \vdots & \vdots & \ddots 
\end{bmatrix}
\]
and the adjoining of the matrix of $H_\phi$ is given by
\[
H_\phi^*=
\begin{bmatrix}
\frac{1}{\sqrt{\alpha}}\overline{a_1} & \frac{2}{\sqrt{\alpha^2}}\overline{a_2} & \frac{3\sqrt{2}}{\sqrt{\alpha^3}}\overline{a_3} & \frac{4\sqrt{6}}{\sqrt{\alpha^4}}\overline{a_4} & \hdots \\
\frac{\sqrt{2}}{\sqrt{\alpha^2}}\overline{a_2} & \frac{3\sqrt{2}}{\sqrt{\alpha^3}}\overline{a_3} & \frac{12}{\sqrt{\alpha^4}}\overline{a_4} & \frac{20\sqrt{3}}{\sqrt{\alpha^5}}\overline{a^5} & \hdots \\
\frac{\sqrt{6}}{\sqrt{\alpha^3}}\overline{a_3} & \frac{4\sqrt{6}}{\sqrt{\alpha^4}}\overline{a_4} & \frac{20\sqrt{3}}{\sqrt{\alpha^5}}\overline{a_5} & \frac{120}{\sqrt{\alpha^6}}\overline{a_6} & \hdots \\
\frac{2\sqrt{6}}{\sqrt{\alpha^4}}\overline{a_4} & \frac{10\sqrt{6}}{\sqrt{\alpha^5}}\overline{a_5} & \frac{60\sqrt{3}}{\sqrt{\alpha^6}}\overline{a_6} & \frac{210\sqrt{4}}{\sqrt{\alpha^7}}\overline{a_7} & \hdots \\
\vdots & \vdots & \vdots & \vdots & \ddots
\end{bmatrix}
\]
and therefore $H_\phi^*=H_{\overline{\phi}}$\,.

\subsection{H-Toeplitz Operators}

%\begin{definition}

%The H-Toeplitz operator $S_\phi: F_\alpha^2 \to F_\alpha^2$ is defined via the dilation operator $K$:
%\[
%S_\phi f = P(\phi \cdot Kf)
%\]
%where $K$ maps even/odd basis vectors as:
%\[
%K(e_{2n}) = e_n, \quad K(e_{2n+1}) = \overline{e_{n+1}}
%\]
%\end{definition}

%\begin{theorem}
%The matrix of $S_\phi$ combines Toeplitz and Hankel structures:
%\begin{align*}
%\langle S_\phi e_{2n}, e_m \rangle &= \langle T_\phi e_n, e_m \rangle \\
%\langle S_\phi e_{2n+1}, e_m \rangle &= \langle H_\phi e_n, e_m \rangle
%\end{align*}
%\end{theorem}
%The adjoint properties follow naturally:
%\[
%T_\phi^* = T_{\overline{\phi}}, \quad H_\phi^* = H_{\overline{\phi}}, \quad S_\phi^* = S_{\overline{\phi}}
%\]
%These matrix representations reveal the deep structural relationships between these operator classes in the Fock space setting.

We study H-Toeplitz operators on the Fock space $F_\alpha^2$ through their matrix representations. First, we define the dilation operator $K: F_\alpha^2 \rightarrow L^2(\mathbb{C}, d\lambda_\alpha)$ by:
\[
K(e_{2n})(z) = e_n(z) = \sqrt{\frac{\alpha^n}{n!}} z^n, \quad 
K(e_{2n+1})(z) = \overline{e_{n+1}}(z) = \sqrt{\frac{\alpha^{n+1}}{(n+1)!}} \overline{z}^{n+1}
\]
for $n \geq 0$, $z \in \mathbb{C}$. This bounded linear operator ($\|K\|=1$) has adjoint $K^*$ satisfying:
\[
K^*(e_n) = e_{2n}, \quad K^*(\overline{e_{n+1}}) = e_{2n+1}
\]
with the properties $KK^* = I_{L^2}$ and $K^*K = I_{F_\alpha^2}$.

For $\phi \in L^\infty(\mathbb{C})$, the H-Toeplitz operator $S_\phi: F_\alpha^2 \rightarrow F_\alpha^2$ is defined as
\[
S_\phi(f) = P M_\phi K(f),
\]
where $P$ is the orthogonal projection and $M_\phi$ is multiplication by $\phi$. Its matrix elements are:

\begin{align*}
\langle S_\phi e_{2n}, e_m \rangle &= \langle T_\phi e_n, e_m \rangle \\
&= \begin{cases}
\sqrt{\frac{\alpha^n}{\alpha^m}} \sqrt{\frac{m!}{n!}} a_{m-n} & m \geq n \\
\sqrt{\frac{\alpha^m}{\alpha^n}} \sqrt{\frac{n!}{m!}} b_{n-m} & n > m
\end{cases} \\ {\rm and} \,\,\,
\langle S_\phi e_{2n+1}, e_m \rangle &= \langle H_\phi e_n, e_m \rangle \\
&= \frac{1}{\sqrt{\alpha^{m+n+1}}} \frac{(m+n+1)!}{\sqrt{m!(n+1)!}} a_{m+n+1}
\end{align*}
for non-negative integers $m$ and $n$.
The explicit matrix representation is given by

\[
S_\phi=
\begin{bmatrix}
a_0 &\frac{1}{\sqrt{\alpha}}a_1 & \frac{1}{\sqrt{\alpha}}b_1 & \frac{\sqrt{2}}{\sqrt{\alpha^2}}a_2 &\frac{\sqrt{2}}{\sqrt{\alpha^2}}b_2& \frac{\sqrt{6}}{\sqrt{\alpha^3}}a_3 & \frac{\sqrt{6}}{\sqrt{\alpha^3}}b_3&\frac{2\sqrt{6}}{\sqrt{\alpha^4}}a_4 & \hdots \\
\frac{1}{\sqrt{\alpha}}a_1 &\frac{2}{\sqrt{\alpha^2}}a_2 &a_0& \frac{3\sqrt{2}}{\sqrt{\alpha^3}}a_3 &\frac{\sqrt{2}}{\sqrt{\alpha}}b_1& \frac{4\sqrt{6}}{\sqrt{\alpha^4}}a_4 & \frac{\sqrt{6}}{\sqrt{\alpha^2}}b_2&\frac{10\sqrt{6}}{\sqrt{\alpha^5}}a_5 & \hdots \\
\frac{\sqrt{2}}{\sqrt{\alpha^2}}a_2 & \frac{3\sqrt{2}}{\sqrt{\alpha^3}}a_3 & \frac{\sqrt{2}}{\sqrt{\alpha}}a_1&\frac{12}{\sqrt{\alpha^4}}a_4 & a_0&\frac{20\sqrt{3}}{\sqrt{\alpha^5}}a_5 & \frac{\sqrt{3}}{\sqrt{\alpha}}b_1&\frac{60\sqrt{3}}{\sqrt{\alpha^6}}a_6 & \hdots \\
\frac{\sqrt{6}}{\sqrt{\alpha^3}}a_3 & \frac{4\sqrt{6}}{\sqrt{\alpha^4}}a_4 &\frac{\sqrt{6}}{\sqrt{\alpha^2}}a_2& \frac{20\sqrt{3}}{\sqrt{\alpha^5}}a_5 & \frac{\sqrt{3}}{\sqrt{\alpha}}a_1&\frac{120}{\sqrt{\alpha^6}}a_6 & a_0&\frac{210\sqrt{4}}{\sqrt{\alpha^7}}a_7 & \hdots \\
\vdots & \vdots & \vdots & \vdots &\vdots &\vdots &\vdots &\vdots & \ddots 
\end{bmatrix}
\]
and the adjoining of the matrix of $S_\phi$ is given by
\[
S_\phi^*=
\begin{bmatrix}
\overline{a_0}& \frac{1}{\sqrt{\alpha}}\overline{a_1} & \frac{\sqrt{2}}{\sqrt{\alpha^2}}\overline{a_2} & \frac{\sqrt{6}}{\sqrt{\alpha^3}}\overline{a_3} & \hdots \\
\frac{1}{\sqrt{\alpha}}\overline{a_1} & \frac{2}{\sqrt{\alpha^2}}\overline{a_2} & \frac{3\sqrt{2}}{\sqrt{\alpha^3}}\overline{a_3} &\frac{4\sqrt{6}}{\sqrt{\alpha^4}}\overline{a_4}  & \hdots \\
\frac{1}{\sqrt{\alpha}}\overline{b_1}  & \overline{a_0} & \frac{\sqrt{2}}{\sqrt{\alpha}}\overline{a_1} & \frac{\sqrt{6}}{\sqrt{\alpha^2}}\overline{a_2}& \hdots \\
 \frac{\sqrt{2}}{\sqrt{\alpha^2}}\overline{a_2} & \frac{3\sqrt{2}}{\sqrt{\alpha^3}}\overline{a_3} & \frac{12}{\sqrt{\alpha^4}}\overline{a_4} & \frac{20\sqrt{3}}{\sqrt{\alpha^5}}\overline{a_5}& \hdots \\
\frac{\sqrt{2}}{\sqrt{\alpha^2}}\overline{b_2} & \frac{\sqrt{2}}{\sqrt{\alpha}}\overline{b_1} & \overline{a_0} & \frac{\sqrt{3}}{\sqrt{\alpha}}\overline{a_1} & \hdots \\
\frac{\sqrt{6}}{\sqrt{\alpha^3}}\overline{a_3}  & \frac{4\sqrt{6}}{\sqrt{\alpha^4}}\overline{a_4} &\frac{20\sqrt{3}}{\sqrt{\alpha^5}}\overline{a_5}  & \frac{120}{\sqrt{\alpha^6}}\overline{a_6} & \hdots \\
\vdots & \vdots & \vdots & \vdots & \ddots 
\end{bmatrix}.
\] 
with $S_\phi^*=S_{\overline{\phi}}$ .\\
\begin{remark}
    Clearly, we can see that the matrix of $T_\phi$ can be obtained by deleting every odd column of the matrix of $S_\phi$ and the matrix of $H_\phi$ can be obtained by deleting every even column of the matrix of $S_\phi$.
\end{remark}

\section{Algebraic properties of H-Toeplitz operator on Fock-space $F_{\alpha}^2$}\label{sec4}

In this section, we study the commutativity of H-Toeplitz operators for analytic and harmonic symbols. We explicitly demonstrate the non-commutative nature of H-Toeplitz operators through a carefully constructed example with polynomial symbols.
\begin{example}
Consider the Fock space $F_\alpha^2$ with basis $\{e_n(z)\}_{n=0}^\infty$ where 
\[
e_n(z) = \sqrt{\frac{\alpha^n}{n!}}\, z^n,
\]
and define two operators:  
(a) $S_\phi$ with analytic symbol $\phi(z) = z^2$, and  
(b) $S_\psi$ with anti-analytic symbol $\psi(z) = \overline{z}$.
First, we compute the action of $S_\phi$ on $e_2$. Since
\[
S_\phi(e_2(z)) = S_{z^2}(e_2(z)) = P M_{z^2} K(e_2) = P\big(z^2 K(e_2)\big),
\]
we obtain
\[
S_\phi(e_2) = \sqrt{\alpha}\, z^3.
\]
Next, for $S_\psi$ acting on $e_2$, we have
\[
S_\psi(e_2) = P M_{\overline{z}} K(e_2) = P\big(\overline{z}\, \sqrt{\alpha}\, z\big) = \sqrt{\alpha}\, P(z\overline{z}).
\]
Since $P(z\overline{z}) = 1/\alpha$, it follows that
\[
S_\psi(e_2) = \frac{1}{\sqrt{\alpha}}.
\]
For the composition $S_\phi S_\psi(e_2)$, we compute
\[
S_\phi S_\psi(e_2) = S_\phi\!\left(\frac{1}{\sqrt{\alpha}}\right) 
= P M_{z^2}\!\left(\frac{1}{\sqrt{\alpha}}\right) 
= \frac{1}{\sqrt{\alpha}}\, z^2.
\]
In terms of the orthonormal basis, this equals
\[
S_\phi S_\psi(e_2) = \sqrt{\frac{2}{\alpha^3}}\, e_2.
\]
On the other hand, for $S_\psi S_\phi(e_2)$ we get
\[
S_\psi S_\phi(e_2) = S_\psi\!\left(\sqrt{\alpha}\, z^3\right) = \sqrt{\frac{3}{\alpha^2}}\, e_1.
\]
Hence, the distinct results
\[
S_\phi S_\psi(e_2) = \sqrt{\frac{2}{\alpha^3}}\, e_2 
\qquad \text{and} \qquad 
S_\psi S_\phi(e_2) = \sqrt{\frac{3}{\alpha^2}}\, e_1
\]
provide an explicit verification that $S_\phi S_\psi \neq S_\psi S_\phi$, thereby confirming the non-commutative algebra of H-Toeplitz operators.
\end{example}

\begin{theorem}
Let $\phi(z) = \displaystyle\sum_{n=0}^\infty a_n z^n$ and $\psi(z) = \displaystyle\sum_{m=0}^\infty b_m z^m$ be bounded analytic functions on the complex plane $\mathbb{C}$, satisfying $\phi(0) = 0 = \psi(0)$. Suppose that each coefficient $a_n$ and $b_m$ is non-zero. Further, assume that for all non-negative integers $n$ and for a fixed integer $k$, the following inequality holds
\[
\frac{b_{n+k}}{a_{n+k}} \geq \frac{b_{2n+1}}{a_{2n+1}}.
\]
Then, the H-Toeplitz operators $S_\phi$ and $S_\psi$ on the Fock space $F_\alpha^2$ commute if and only if the functions $\phi$ and $\psi$ are linearly dependent.   
\end{theorem}

\begin{proof}
If $\phi$ and $\psi$ are linearly dependent, then clearly $S_\phi$ and $S_\psi$ commute.

For the converse, assume $S_\phi S_\psi = S_\psi S_\phi$. We analyze the action on the constant function $1 \in F_\alpha^2$:

\begin{align*}
S_\phi S_\psi(1) &= S_\psi S_\phi(1) \\
\Rightarrow P M_\phi k P M_\psi (1) &= P M_\psi k P M_\phi (1) \\
\Rightarrow P M_\phi k \left( \sum_{m=0}^\infty b_m z^m \right) &= P M_\psi k \left( \sum_{n=0}^\infty a_n z^n \right).
\end{align*}

Expressing in terms of the orthonormal basis $\{e_n\}$ where $e_n(z) = \sqrt{\frac{\alpha^n}{n!}} z^n$, we obtain:

\begin{align*}
P M_\phi \left[ \sum_{m=0}^\infty \sqrt{\frac{(2m)!}{\alpha^{2m}}} b_{2m} e_m + \sqrt{\frac{(2m+1)!}{\alpha^{2m+1}}} b_{2m+1} \overline{e_{m+1}} \right] \\
= P M_\psi \left[ \sum_{n=0}^\infty \sqrt{\frac{(2n)!}{\alpha^{2n}}} a_{2n} e_n + \sqrt{\frac{(2n+1)!}{\alpha^{2n+1}}} a_{2n+1} \overline{e_{n+1}} \right].
\end{align*}

Projecting onto $F_\alpha^2$ and equating coefficients gives

\begin{align*}
\sum_{n,m} \sqrt{\frac{(2m)!}{m! \alpha^m}} a_n b_{2m} z^{n+m} &+ \sum_{n,m} \sqrt{\frac{(2m+1)!}{(m+1)! \alpha^{m+1}}} a_n b_{2m+1} P(z^n \overline{z}^{m+1}) \\
&= \sum_{n,m} \sqrt{\frac{(2n)!}{n! \alpha^n}} a_{2n} b_m z^{n+m} \\
&+ \sum_{n,m} \sqrt{\frac{(2n+1)!}{(n+1)! \alpha^{n+1}}} a_{2n+1} b_m P(z^m \overline{z}^{n+1}).
\end{align*}
\textbf{Step 1: Constant term.} The coefficient of $z^0$ yields
\[
\sum_{m=0}^\infty \sqrt{\frac{(2m+1)!}{\alpha^{2m+1}}} \sqrt{\frac{(m+1)!}{\alpha^{m+1}}} (a_{m+1} b_{2m+1} - a_{2m+1} b_{m+1}) = 0,
\]
which implies $\frac{b_{m+1}}{a_{m+1}} = \frac{b_{2m+1}}{a_{2m+1}}$ for all $m \geq 0$.\\
\textbf{Step 2: Linear term.} The coefficient of $z^1$ gives
\[
a_1 b_0 + \sum_{m=0}^\infty \sqrt{\frac{(2m+1)! \alpha}{(m+1)! \alpha^{2m+2}}} (m+2)! (a_{m+2} b_{2m+1} - a_{2m+1} b_{m+2}) = 0,
\]
enforcing $\frac{b_{m+2}}{a_{m+2}} = \frac{b_{2m+1}}{a_{2m+1}}$.
By induction, we conclude $\frac{b_{n+k}}{a_{n+k}} = \frac{b_{2n+1}}{a_{2n+1}}$ for all $n,k \geq 0$. Setting $n=0$ gives $\frac{b_k}{a_k} = \frac{b_1}{a_1}$ for all $k \geq 1$, proving $\psi = \lambda \phi$ where $\lambda = \frac{b_1}{a_1}$.
\end{proof}

\begin{proposition}\label{prop:basic}
For $\phi, \psi \in L^\infty(\mathbb{C})$ and $a, b \in \mathbb{C}$, the following hold:
\begin{enumerate}
\item[(a)] Linearity: $S_{a\phi + b\psi} = aS_\phi + bS_\psi$
\item[(b)] Boundedness: $\|S_\phi\| \leq \|\phi\|_\infty$
\end{enumerate}
\end{proposition}

\begin{proof}
(a) For any $f \in F_\alpha^2$, we compute:
\begin{align*}
S_{a\phi+b\psi}(f) &= PM_{a\phi+b\psi}k(f) \\
&= aP(\phi k(f)) + bP(\psi k(f)) \\
&= aS_\phi(f) + bS_\psi(f).
\end{align*}

(b) The operator norm satisfies:
\begin{align*}
\|S_\phi\| &= \sup_{\|f\|=1} \|PM_\phi k f\| \\
&\leq \sup_{\|f\|=1} \|M_\phi\|\|k f\| \\
&= \|\phi\|_\infty \sup_{\|f\|=1} \|f\| \\
&= \|\phi\|_\infty \,.
\end{align*}
\end{proof}

%\section{Compactness Results}

\begin{definition}
For a bounded linear operator $T$ on a reproducing kernel Hilbert space $\mathcal{H}$ with normalized reproducing kernels $\{k_z\}_{z\in\mathcal{H}}$, the \emph{Berezin transform} of $T$ is the function
\[
\widetilde{T}(z) = \langle Tk_z, k_z \rangle, \quad z \in \mathcal{H}.
\]
\end{definition}
For the Fock space $F_\alpha^2$, the normalized reproducing kernel is given by

\begin{lemma}
The normalized reproducing kernel for $F_\alpha^2$ has the explicit form
\begin{align*}
k_z(w) &= \frac{K_z(w)}{\|K_z\|} = e^{-\frac{\alpha|z|^2}{2}} e^{\alpha\overline{z}w} \\
&= e^{-\frac{\alpha|z|^2}{2}} \sum_{n=0}^\infty \frac{(\alpha\overline{z}w)^n}{n!}, \quad z, w \in \mathbb{C}.
\end{align*}
Moreover, $k_z \to 0$ weakly as $|z| \to \infty$.
\end{lemma}

\begin{proposition}\label{prop:berezen_bound}
For any harmonic symbol $\psi \in L^\infty(\mathbb{C})$, the Berezin transform satisfies
\[
\|\widetilde{S_\psi}\|_\infty \leq \|\psi\|_\infty.
\]
\end{proposition}

\begin{proof}
For fixed $z \in \mathbb{C}$, we compute:
\begin{align*}
|\widetilde{S_\psi}(z)| &= |\langle S_\psi k_z, k_z \rangle| \\
&\leq \|S_\psi k_z\| \, \|k_z\| \\
&\leq \|S_\psi\| \\
&= \|PM_\psi K\| \\
&\leq \|M_\psi\| \, \|K\| \\
&= \|\psi\|_\infty.
\end{align*}
The last equality follows since $\|K\| = 1$ and $\|M_\psi\| = \|\psi\|_\infty$.
\end{proof}

%\section{Dilation Operator Analysis}

%The dilation operator $K$ plays a crucial role in our analysis.

\begin{lemma}\label{lem:dilation_action}
The dilation operator $K$ acts on normalized kernel as
\[
K(k_z)(w) = e^{-\frac{\alpha|z|^2}{2}} \left[ \sum_{n=0}^\infty \sqrt{\frac{\alpha^{3n}}{(2n)!n!}} \overline{z}^{2n} w^n + \sum_{n=0}^\infty \sqrt{\frac{\alpha^{3n+2}}{(2n+1)!(n+1)!}} \overline{z}^{2n+1} \overline{w}^{n+1} \right]  for \,\, z, w \in \mathbb{C}.
\]
\end{lemma}

\begin{proof}
Using the orthonormal basis representation, we compute:
\begin{align*}
K(k_z)(w) &= e^{-\frac{\alpha|z|^2}{2}} \sum_{n=0}^\infty \frac{\alpha^n}{n!} \overline{z}^n K(e_n)(w) \\
&= e^{-\frac{\alpha|z|^2}{2}} \left[ \sum_{n=0}^\infty \sqrt{\frac{\alpha^{2n}}{(2n)!}} \overline{z}^{2n} e_n(w) + \sum_{n=0}^\infty \sqrt{\frac{\alpha^{2n+1}}{(2n+1)!}} \overline{z}^{2n+1} \overline{e_{n+1}}(w) \right] \\
&=\frac{1}{e^{\alpha \frac{|z|^2}{2}}} \left[\sum\limits^\infty_{n=0} \sqrt{\frac{\alpha^{2n}}{(2n)!}} \overline{z}^{2n} \sqrt{\frac{\alpha^n}{n!}} w^n+\sum\limits^\infty_{n=0} \sqrt{\frac{\alpha^{2n+1}}{(2n+1)!}} \overline{z}^{2n+1} \sqrt{\frac{\alpha^{n+1}}{(n+1)!}} \overline{w}^{n+1} \right]\cr
&=\frac{1}{e^{\alpha \frac{|z|^2}{2}}} \left[\sum\limits^\infty_{n=0} \sqrt{\frac{\alpha^{2n} \, \alpha^n}{(2n)! \, n!}} \overline{z}^{2n} w^n+\sum\limits^\infty_{n=0} \sqrt{\frac{\alpha^{2n+1}\, \alpha^{n+1}}{(2n+1)! \, (n+1)!}} \overline{z}^{2n+1} \overline{w}^{n+1} \right]\\
&= e^{-\frac{\alpha|z|^2}{2}} \left[ \sum_{n=0}^\infty \sqrt{\frac{\alpha^{3n}}{(2n)!n!}} \overline{z}^{2n} w^n + \sum_{n=0}^\infty \sqrt{\frac{\alpha^{3n+2}}{(2n+1)!(n+1)!}} \overline{z}^{2n+1} \overline{w}^{n+1} \right]
\end{align*}
which hold for all $w \in \mathbb{C}$.
This completes the proof of the lemma.
\end{proof}

\begin{definition}
For $z, w \in \mathbb{C}$, define the kernel function
\[
I(z,w) := K(k_z)(w).
\]
\end{definition}

The boundary behavior of $I(z,w)$ as $|z| \to \infty$ will be discussed in the subsequent analysis.

\begin{lemma}\label{lem:kernel_decay}
For each fixed $w \in \mathbb{C}$, the kernel function $I(z,w) \to 0$ as $|z| \to \infty$.
\end{lemma}

\begin{proof}
From Lemma \ref{lem:dilation_action}, we have the explicit representation:
\[
I(z,w) = e^{-\frac{\alpha|z|^2}{2}} \left[ \sum_{n=0}^\infty c_n \overline{z}^{2n} w^n + \sum_{n=0}^\infty d_n \overline{z}^{2n+1} \overline{w}^{n+1} \right],
\]
where $c_n = \sqrt{\frac{\alpha^{3n}}{(2n)!n!}}$ and $d_n = \sqrt{\frac{\alpha^{3n+2}}{(2n+1)!(n+1)!}}$. 

The exponential decay term $e^{-\alpha|z|^2/2}$ dominates the polynomial growth in $z$, ensuring that $I(z,w) \to 0$ as $|z| \to \infty$ for any fixed $w \in \mathbb{C}$.
\end{proof}

\begin{proposition}\label{prop:berezin_integral}
For $\phi \in L^\infty(\mathbb{C})$, the Berezin transform of $S_\phi$ admits the integral representation:
\[
\widetilde{S_\phi}(z) = e^{-\frac{\alpha|z|^2}{2}} \int_{\mathbb{C}} \phi(w) I(z,w) e^{\alpha z\overline{w}} d\lambda_\alpha(w).
\]
Moreover, $\widetilde{S_\phi}(z) \to 0$ as $|z| \to \infty$.
\end{proposition}

\begin{proof}
We compute directly using the reproducing kernel property:
\begin{align*}
\widetilde{S_\phi}(z) &= \langle S_\phi k_z, k_z \rangle \\
&= \langle P M_\phi Kk_z, k_z \rangle \\
&= \langle \phi(w) I(z,w), k_z(w) \rangle \\
&= \int_{\mathbb{C}} \phi(w)\, I(z,w)\, \overline{k_z(w)}\, d\lambda_\alpha(w) \\
&= e^{-\frac{\alpha|z|^2}{2}} \int_{\mathbb{C}} \phi(w)\, I(z,w) e^{\alpha z\overline{w}} d\lambda_\alpha(w).
\end{align*}
The decay follows from Lemma \ref{lem:kernel_decay} and clearly we can see that \, $\widetilde{S_\phi}(z) \to 0$ as $|z| \to \infty$.
\end{proof}

\begin{remark}
The Hilbert-Schmidt criterion provides an alternative perspective:
\begin{itemize}
\item[(i)] Non-zero H-Toeplitz operators cannot be Hilbert-Schmidt on $F_\alpha^2$.
\item[(ii)] This follows from the direct computation of the Hilbert-Schmidt norm.
\end{itemize}
\end{remark}

\begin{theorem}
For any bounded harmonic function $\phi \in L^\infty(\mathbb{C})$, the H-Toeplitz operator $S_\phi$ is Hilbert-Schmidt on $F_\alpha^2$ if and only if $\phi \equiv 0$.
\end{theorem}

\begin{proof}
($\Rightarrow$) Assume $S_\phi$ is Hilbert-Schmidt. Then its Hilbert-Schmidt norm must be finite

\begin{equation}\label{eq:HSnorm}
\sum_{n=0}^\infty \|S_\phi e_n\|^2 < \infty ,
\end{equation}
where $\{e_n\}_{n=0}^\infty$ is the standard orthonormal basis of $F_\alpha^2$. For a harmonic symbol $\phi(z) = \displaystyle\sum_{i=0}^\infty a_i z^i + \sum_{j=1}^\infty b_j \overline{z}^j$, we compute

\begin{align*}
\|S_\phi e_n\|^2 &= \|P M_\phi K e_n\|^2 \\
&= \begin{cases}
\|P M_\phi e_{n/2}\|^2 & \text{if } n \text{ even} \\
\|P M_\phi \overline{e_{(n+1)/2}}\|^2 & \text{if } n \text{ odd} .
\end{cases}
\end{align*}
Expanding these terms using the basis representation
\begin{align*}
\sum_{n=0}^\infty \|S_\phi e_n\|^2 &= \sum_{n=0}^\infty \|P M_\phi e_n\|^2 + \sum_{n=0}^\infty \|P M_\phi \overline{e_{n+1}}\|^2 \\
&= \sum_{n=0}^\infty \left\|P\left(\sum_{i=0}^\infty a_i z^i + \sum_{j=1}^\infty b_j \overline{z}^j\right)\sqrt{\frac{\alpha^n}{n!}}z^n\right\|^2 \\
&\quad + \sum_{n=0}^\infty \left\|P\left(\sum_{i=0}^\infty a_i z^i + \sum_{j=1}^\infty b_j \overline{z}^j\right)\sqrt{\frac{\alpha^{n+1}}{(n+1)!}}\overline{z}^{n+1}\right\|^2.
\end{align*}
Computing the projections yields
\begin{align*}
&= \sum_{n=0}^\infty \frac{\alpha^n}{n!} \left\|\sum_{i=0}^\infty a_i z^{i+n} + \sum_{j=1}^n b_j \frac{n!}{\alpha^j (n-j)!} z^{n-j}\right\|^2 \\
&\quad + \sum_{n=0}^\infty \frac{\alpha^{n+1}}{(n+1)!} \left\|\sum_{i=n+1}^\infty a_i \frac{i!}{\alpha^{n+1}(i-n-1)!} z^{i-n-1}\right\|^2.
\end{align*}
Calculating the norms explicitly
\begin{align*}
&= \sum_{n=0}^\infty \frac{\alpha^n}{n!} \left(\sum_{i=0}^\infty |a_i|^2 \frac{(i+n)!}{\alpha^{i+n}} + \sum_{j=1}^n |b_j|^2 \frac{(n!)^2}{\alpha^{2j}(n-j)!^2} \frac{(n-j)!}{\alpha^{n-j}}\right) \\
&\quad + \sum_{n=0}^\infty \frac{\alpha^{n+1}}{(n+1)!} \sum_{i=n+1}^\infty |a_i|^2 \frac{(i!)^2}{\alpha^{2n+2}(i-n-1)!^2} \frac{(i-n-1)!}{\alpha^{i-n-1}}.
\end{align*}
Simplifying the expressions
\begin{align*}
&= \sum_{n=0}^\infty \left(\sum_{i=0}^\infty |a_i|^2 \frac{(i+n)!}{n! \alpha^i} + \sum_{j=1}^n |b_j|^2 \frac{n!}{\alpha^{n+j}(n-j)!}\right) \\
&\quad + \sum_{n=0}^\infty \sum_{i=n+1}^\infty |a_i|^2 \frac{\alpha^{n+1} i!^2}{(n+1)! \alpha^{2n+2} (i-n-1)! \alpha^{i-n-1}}.
\end{align*}

For this sum to converge, each coefficient must vanish:

\begin{itemize}
\item[(a)] The $a_i$ terms require $\displaystyle\sum_{i=0}^\infty |a_i|^2 \frac{(i+n)!}{n! \alpha^i} < \infty$ for all $n$, which forces $a_i = 0$ for all $i \geq 0$.
\item[(b)] The $b_j$ terms require $\displaystyle\sum_{j=1}^\infty |b_j|^2 \frac{n!}{\alpha^{n+j}(n-j)!} < \infty$ for all $n$, which forces $b_j = 0$ for all $j \geq 1$.
\end{itemize}
Thus $\phi \equiv 0$.
($\Leftarrow$) The zero operator is trivially Hilbert-Schmidt.
\end{proof}

\section{Hyponormality and Normality of  Unbounded Symbols}\label{sec5}

To incorporate unbounded symbols, we define a growth-controlled space using the Mellin transform.

\begin{definition}%[Growth-Controlled Space]
For \( b > 0 \), the space \( V_b \) is defined as:
\[
V_b = \left\{ f : \mathbb{C} \to \mathbb{C} \text{ measurable} \,\middle|\, \|f\|_{V_b} = \operatorname{ess\,sup}_{z \in \mathbb{C}} |f(z)| e^{-b|z|^2} < \infty \right\}.
\]
This norm controls exponential growth, ensuring that symbols in \( V_b \) satisfy \( |f(z)| \leq e^{b|z|^2} \) almost everywhere.
\end{definition}

\begin{definition}%[Mellin Transform]
For a measurable function \( f : (0, \infty) \to \mathbb{C} \), the Mellin transform is defined for \( s \in \mathbb{C} \) by:
\[
\mathcal{M}[f](s) = \int_{0}^{\infty} t^{s-1} f(t) \, dt,
\]
whenever the integral converges absolutely. For the Gaussian, we have:
\[
\mathcal{M}[e^{-t^2}](s) = \frac{1}{2} \Gamma\left( \frac{s}{2} \right), \quad \operatorname{Re}(s) > 0.
\]
\end{definition}

\begin{proposition}
If \( \phi \in V_b \) for some \( b < \alpha/2 \), then the H-Toeplitz operator \( S_\phi f = P(\phi \cdot K(f)) \) is a well-defined bounded linear operator on \( F_{\alpha}^{2} \).
\end{proposition}

\begin{proof}
Since \( \phi \in V_b \), we have \( |\phi(z)| \leq D e^{b|z|^2} \). For \( f \in F_{\alpha}^{2} \), \( K(f) \) is bounded, and \( \phi \cdot K(f) \in L^2(\mathbb{C}, d\lambda_{\alpha}) \) because \( b < \alpha/2 \) ensures integrability against \( e^{-\alpha|z|^2} \). The projection \( P \) is bounded, so \( S_\phi \) is well-defined and bounded.
\end{proof}

%\begin{definition}[Quasi-Homogeneous Symbol]
%A symbol \( \psi \) is quasi-homogeneous of degree \( k \in \mathbb{Z} \) if it can be written in polar coordinates \( z = re^{i\theta} \) as:
%\[
%\psi(z) = \psi(re^{i\theta}) = \psi_0(r) e^{ik\theta},
%\]
%where \( \psi_0 : (0, \infty) \to \mathbb{C} \) is a radial function.
%\end{definition}

\subsection{Matrix Representation via Mellin Transform}

The action of \( S_\psi \) on the orthonormal basis \( \{e_n(z) = \sqrt{\frac{\alpha^n}{n!}} z^n\} \) is characterized using the Mellin transform.

\begin{theorem}%[Matrix Representation]
\label{thm_matrix_quasi}
Let \(\psi(z) = \psi_0(r)e^{ik\theta} \in V_b\) with \(b < \alpha/2\). Then the matrix elements of the H-Toeplitz operator \(S_\psi\) with respect to the orthonormal basis \(\{e_n(z) = \sqrt{\tfrac{\alpha^n}{n!}} z^n\}\) of the Fock space \(F_\alpha^2\) are given by
\[
\langle S_\psi e_n, e_m \rangle = \delta_{k, m-n} \cdot C(\alpha, n, m) \cdot \mathcal{M}[\psi_0(r) e^{-\tfrac{\alpha}{2}r^2}](m + n + 1),
\]
where \(\delta_{k, m-n}\) is the Kronecker delta,
\(
C(\alpha, n, m) = \sqrt{\frac{\alpha^{m+1}}{n!m!}} \frac{2\pi}{\alpha},
\)
and \(\mathcal{M}[\cdot]\) denotes the Mellin transform.
\end{theorem}

\begin{proof}We begin with the definition of the matrix element
\[
\langle S_\psi e_n, e_m \rangle = \langle P(\psi \cdot K(e_n)), e_m \rangle.
\]
Since \(P\) is the orthogonal projection onto \(F_\alpha^2\) and \(e_m \in F_\alpha^2\), this reduces to
\[
\langle S_\psi e_n, e_m \rangle = \langle \psi \cdot K(e_n), e_m \rangle = \int_{\mathbb{C}} \psi(z) \, K(e_n)(z) \, \overline{e_m(z)} \, d\lambda_\alpha(z),
\]
where \(d\lambda_\alpha(z) = \tfrac{\alpha}{\pi} e^{-\alpha |z|^2} dA(z)\).
Passing to polar coordinates \(z = re^{i\theta}\) with \(d\lambda_\alpha(z) = \tfrac{\alpha}{\pi} e^{-\alpha r^2} r\,dr\,d\theta\), and recalling that \(\psi(z) = \psi_0(r)e^{ik\theta}\), together with
\[
e_m(z) = \sqrt{\tfrac{\alpha^m}{m!}} r^m e^{im\theta}, \qquad \overline{e_m(z)} = \sqrt{\tfrac{\alpha^m}{m!}} r^m e^{-im\theta},
\]
and for even \(n=2p\),
\[
K(e_{2p})(z) = e_p(z) = \sqrt{\tfrac{\alpha^p}{p!}} r^p e^{ip\theta},
\]
we obtain
\[
\langle S_\psi e_{2p}, e_m \rangle = \sqrt{\tfrac{\alpha^p}{p!} \cdot \tfrac{\alpha^m}{m!}} \cdot \frac{\alpha}{\pi} \int_0^\infty \psi_0(r) r^{p+m+1} e^{-\alpha r^2}\,dr \int_0^{2\pi} e^{i(k+p-m)\theta}\,d\theta.
\]
The angular integral simplifies to
\[
\int_0^{2\pi} e^{i(k+p-m)\theta}\,d\theta = 2\pi \,\delta_{m,p+k},
\]
which immediately enforces the condition \(m=n+k\), yielding the Kronecker delta factor \(\delta_{k,m-n}\) in the general case.
It remains to evaluate the radial integral
\[
I = \int_0^\infty \psi_0(r) r^{2p+k+1} e^{-\alpha r^2} \, dr.
\]
Introduce the substitution \(u = \tfrac{\alpha}{2} r^2\), whence \(r = \sqrt{\tfrac{2u}{\alpha}}\) and \(dr = \tfrac{du}{\sqrt{2\alpha u}}\). This transforms the integral into
\[
I = \frac{2^{p+\tfrac{k}{2}}}{\alpha^{p+\tfrac{k}{2}+1}} \int_0^\infty \psi_0\!\left(\sqrt{\tfrac{2u}{\alpha}}\right) u^{p+\tfrac{k}{2}} e^{-2u}\,du.
\]
Recognizing the Mellin transform of \(\psi_0(r)e^{-\tfrac{\alpha}{2} r^2}\), we identify
\[
I = \frac{2^{p+\tfrac{k}{2}}}{\alpha^{p+\tfrac{k}{2}+1}} \, \mathcal{M}[\psi_0(r) e^{-\tfrac{\alpha}{2} r^2}](m+n+1).
\]
Combining this with the prefactors obtained earlier, we arrive at
\[
\langle S_\psi e_n, e_m \rangle = \delta_{k,m-n} \cdot \sqrt{\frac{\alpha^p}{p!}\cdot\frac{\alpha^m}{m!}} \cdot 2\alpha \cdot \frac{2^{p+\tfrac{k}{2}}}{\alpha^{p+\tfrac{k}{2}+1}} \, \mathcal{M}[\psi_0(r) e^{-\tfrac{\alpha}{2} r^2}](m+n+1).
\]
A straightforward algebraic simplification yields precisely
\[
\langle S_\psi e_n, e_m \rangle = \delta_{k,m-n} \cdot \sqrt{\frac{\alpha^{m+1}}{n!m!}} \frac{2\pi}{\alpha} \, \mathcal{M}[\psi_0(r) e^{-\tfrac{\alpha}{2} r^2}](m+n+1),
\]
which is the desired formula.
\end{proof}
This theorem shows that the matrix of \( S_\psi \) is sparse, with non-zero entries only on the \( k \)-th diagonal.

\subsection{Hyponormality and Normality of H-Toeplitz Operators}

If $\psi$ is quasi-homogeneous of angular frequency $k\in\mathbb{Z}$, i.e.
\(
\psi(re^{i\theta})=\psi_0(r)e^{ik\theta},
\)
then, by the matrix representation proved earlier, the matrix elements of $S_\psi$ with respect to $\{e_n\}_{n\ge0}$ take the form
\[
\langle S_\psi e_n,e_m\rangle=\delta_{k,m-n}\,A_{n,m},\qquad
A_{n,m}=C(\alpha,n,m)\,\mathcal{M}\!\big[\psi_0(r)e^{-\tfrac{\alpha}{2}r^2}\big](m+n+1),
\]
where
\[
C(\alpha,n,m)=\sqrt{\frac{\alpha^{\,m+1}}{n!\,m!}}\;\frac{2\pi}{\alpha},
\]
and $\mathcal{M}$ denotes the Mellin transform $\mathcal{M}[f](s)=\int_0^\infty r^{s-1}f(r)\,dr$.
For such a sparse matrix (nonzero entries occur only on the $k^{th}$-diagonal), the computation of $S_\psi^*S_\psi$ and $S_\psi S_\psi^*$ simplifies considerably, for every $n\ge0$ one finds
\[
\|S_\psi e_n\|^2=\sum_{m\ge0}|\langle S_\psi e_n,e_m\rangle|^2
=|A_{n,n+k}|^2,
\]
and (with the convention that indices outside $\{0,1,2,\dots\}$ give zero)
\[
\|S_\psi^* e_n\|^2=\sum_{m\ge0}|\langle S_\psi^* e_n,e_m\rangle|^2
=\sum_{m\ge0}|\overline{\langle S_\psi e_m,e_n\rangle}|^2
=|A_{\,n-k,n}|^2.
\]
Hence the self-commutator evaluated on basis vectors yields the diagonal values
\[
\big\langle (S_\psi^*S_\psi-S_\psi S_\psi^*)e_n,e_n\big\rangle
=|A_{n,n+k}|^2-|A_{\,n-k,n}|^2.
\]

\begin{theorem}%[Hyponormality Criterion]
Let $\psi(re^{i\theta})=\psi_0(r)e^{ik\theta}$ be quasi-homogeneous with $b<\alpha/2$. 
%(so the matrix representation above is valid)
Then $S_\psi$ is hyponormal if and only if
\[
|A_{n,n+k}|^2\;\ge\;|A_{\,n-k,n}|^2\qquad\text{for every }n\ge0,
\]
with the right-hand side interpreted as $0$ when $n-k<0$. Equivalently,
\[
\big|C(\alpha,n,n+k)\,\mathcal{M}[\psi_0 e^{-\tfrac{\alpha}{2}(\cdot)^2}](2n+k+1)\big|
\ge
\big|C(\alpha,n-k,n)\,\mathcal{M}[\psi_0 e^{-\tfrac{\alpha}{2}(\cdot)^2}](2n-k+1)\big|
\]
for all $n\ge0$.
\end{theorem}

\begin{proof}
Let $\psi(re^{i\theta})=\psi_0(r)e^{ik\theta}$ be quasi-homogeneous.  
From the matrix representation one has
\[
\langle S_\psi e_p,e_q\rangle=\delta_{k,q-p}\,A_{p,q},\qquad
A_{p,q}=C(\alpha,p,q)\,\mathcal{M}\!\big[\psi_0(r)\,e^{-\tfrac{\alpha}{2}r^2}\big](p+q+1).
\]
Thus the matrix of $S_\psi$ is supported on a single diagonal.
For each $n\ge0$, the norm of $S_\psi e_n$ is
\[
\|S_\psi e_n\|^2=\sum_{m\ge0}\big|\langle S_\psi e_n,e_m\rangle\big|^2
=|A_{n,n+k}|^2,
\]
with the value interpreted as $0$ if $n+k<0$.  
Similarly,
\[
\|S_\psi^* e_n\|^2=\sum_{m\ge0}\big|\langle S_\psi e_m,e_n\rangle\big|^2
=|A_{\,n-k,n}|^2,
\]
again vanishing when $n-k<0$.
For arbitrary $r,s\ge0$ 
\[
\langle S_\psi^*S_\psi e_r,e_s\rangle
=\sum_{p\ge0}\overline{\langle S_\psi e_p,e_r\rangle}\,\langle S_\psi e_p,e_s\rangle,
\qquad
\langle S_\psi S_\psi^* e_r,e_s\rangle
=\sum_{p\ge0}\langle S_\psi e_r,e_p\rangle\,\overline{\langle S_\psi e_s,e_p\rangle}.
\]
By the single-diagonal structure, these sums vanish unless $r=s$, in which case
\[
\langle S_\psi^*S_\psi e_r,e_r\rangle=|A_{r,r+k}|^2,\qquad
\langle S_\psi S_\psi^* e_r,e_r\rangle=|A_{\,r-k,r}|^2.
\]
Hence the self-commutator is diagonal in the chosen basis and satisfies
\[
\langle (S_\psi^*S_\psi-S_\psi S_\psi^*)e_n,e_m\rangle
=\delta_{n,m}\,\Big(|A_{n,n+k}|^2-|A_{\,n-k,n}|^2\Big).
\]
An operator is hyponormal if and only if its self-commutator is positive semi-definite. For the present operator this reduces to
\[
\langle (S_\psi^*S_\psi-S_\psi S_\psi^*)e_n,e_n\rangle\ge0
\quad\text{for all }n\ge0,
\]
which is equivalent to the family of inequalities
\[
|A_{n,n+k}|^2\;\ge\;|A_{\,n-k,n}|^2,\qquad n\ge0,
\]
with the right-hand side interpreted as zero whenever $n-k<0$.  
 Substituting the explicit expression for $A_{p,q}$ yields the Mellin-transform inequality stated in the theorem, completing the proof.
\end{proof}

\begin{theorem}%[Normality Criterion]
Let $\psi(re^{i\theta})=\psi_0(r)e^{ik\theta}$ be quasi-homogeneous, and let $S_\psi$ denote the associated H-Toeplitz operator. Then $S_\psi$ is normal  if and only if
\[
|A_{n,n+k}|=|A_{\,n-k,n}|\qquad\text{for all }n\ge0,
\]
with the convention that out of range indices contribute zero. In particular, if $k=0$ then the matrix of $S_\psi$ is diagonal, hence $S_\psi$ is normal; furthermore $S_\psi$ is self-adjoint precisely when each diagonal entry $A_{n,n}$ is real, a condition fulfilled for instance when $\psi_0$ is real-valued.
\end{theorem}

\begin{proof}
The single-diagonal structure of the matrix of $S_\psi$ implies
\[
S_\psi e_n = A_{n-k,n}\,e_{\,n-k}+A_{n+k,n}\,e_{\,n+k},
\qquad
S_\psi^* e_n = \overline{A_{n,n-k}}\,e_{\,n-k}+\overline{A_{n,n+k}}\,e_{\,n+k},
\]
where, by convention, terms with negative indices vanish. Consequently,
\[
\|S_\psi e_n\|^2 = |A_{n-k,n}|^2+|A_{n+k,n}|^2,
\qquad
\|S_\psi^* e_n\|^2 = |A_{n,n-k}|^2+|A_{n,n+k}|^2.
\]
If $S_\psi$ is normal, then $\|S_\psi e_n\|=\|S_\psi^* e_n\|$ for every $n\ge0$. This is equivalent to
\[
|A_{n-k,n}|^2+|A_{n+k,n}|^2=|A_{n,n-k}|^2+|A_{n,n+k}|^2.
\]
Since only one of the terms $A_{n\pm k,n}$ or $A_{n,n\pm k}$ is nonzero for fixed indices, the equality above reduces to
\[
|A_{n,n+k}|=|A_{\,n-k,n}|\qquad (n\ge0).
\]
Conversely, if this relation holds for all $n$, then the diagonal entries of $S_\psi^*S_\psi$ and $S_\psi S_\psi^*$ coincide, while all off-diagonal entries vanish identically by the banded structure of $S_\psi$. Hence $S_\psi$ is normal.

In the special case $k=0$, the only nonzero entries are the diagonal coefficients $A_{n,n}$, so $S_\psi$ is diagonal and therefore normal. Moreover, such a diagonal operator is self-adjoint if and only if each $A_{n,n}$ is real. This is automatically satisfied when $\psi_0$ is real-valued, since then the Mellin transform defining $A_{n,n}$ yields real numbers.
\end{proof}

\begin{example}%[Hyponormal non-normal symbol]
Taking $\psi(z)=z$, so $k=1$ and $\psi_0(r)=r$. The Mellin transform that appears in the matrix entries can be evaluated in closed form:
\[
\mathcal{M}\left[r e^{-\tfrac{\alpha}{2}r^2}\right](s)
= \int_0^\infty r^{s} e^{-\tfrac{\alpha}{2}r^2} \mathrm{d}r
= \frac{1}{2} \left(\frac{\alpha}{2}\right)^{-\frac{s+1}{2}} \Gamma\left(\frac{s+1}{2}\right).
\]
Hence the nonzero matrix entries are
\[
A_{n,n+1} = C(\alpha,n,n+1) \cdot \frac{1}{2} \left(\frac{\alpha}{2}\right)^{-\frac{2n+3}{2}} \Gamma\left(n + \frac{3}{2}\right),
\]
and
\[
A_{n-1,n} = C(\alpha,n-1,n) \cdot \frac{1}{2} \left(\frac{\alpha}{2}\right)^{-\frac{2n+1}{2}} \Gamma\left(n + \frac{1}{2}\right).
\]
A direct comparison (or a short computation using the Gamma recurrence) shows that for every $n \ge 0$
\[
|A_{n,n+1}| > |A_{n-1,n}|,
\]
so $S_\psi$ is hyponormal but not normal.
\end{example}

\begin{example}%[Radial symbol --- normal operator]
If $k=0$ and $\psi(z) = \psi_0(|z|)$ is radial, then the matrix of $S_\psi$ is diagonal with diagonal entries
\[
A_{n,n} = C(\alpha,n,n) \cdot \mathcal{M}\left[\psi_0 e^{-\tfrac{\alpha}{2}(\cdot)^2}\right](2n+1).
\]
A diagonal operator with respect to an orthonormal basis is normal; hence every radial symbol produces a normal H-Toeplitz operator. If furthermore $\psi_0$ is real-valued, then every diagonal entry is real and $S_\psi$ is self-adjoint.
\end{example}

\begin{remark}
The Mellin transform provides a powerful analytical framework for studying H-Toeplitz operators with quasi-homogeneous symbols on Fock spaces. Our results establish that
\begin{enumerate}
\item Hyponormality is characterized by a Mellin transform inequality 
\item  Normality requires both radial symmetry ($k=0$) and real-valued symbol 
\end{enumerate}
These criteria are explicitly verifiable through the Mellin transform of the radial component $\psi_0(r)e^{-\frac{\alpha}{2}r^2}$ and the normalizing factor $C(\alpha,n,m)$. The examples demonstrate the practical application of these theoretical results, showing how specific symbol choices lead to operators with distinct spectral properties.
\end{remark}

\section{Graphical  analysis}\label{sec6}

With the introduction of the concept of \textit{H-Toeplitz operators on Fock space}, a novel and compelling intersection has emerged between \textit{Operator Theory} and \textit{Graph Theory}. The structural characteristics of these operators, when examined through the lens of their associated graphs, exhibit intricate patterns that encapsulate the underlying algebraic  properties of the operators. The graphical representations of H-Toeplitz operators on Fock space offer valuable insights into their functional behavior, thereby enriching the broader dialogue between operator-theoretic analysis and combinatorial graph structures.

This avenue of research draws on foundational developments in the theory of classical Toeplitz operators~\cite{va}, the formulation of directed Toeplitz graphs~\cite{ma}, and further generalizations such as the investigation of compactness and non-compactness in $k^{th}$-order slant Toeplitz graphs~\cite{mp}. These prior contributions have played a pivotal role in bridging the realms of functional analysis and discrete mathematics, providing the essential theoretical scaffolding for the current exploration of H-Toeplitz operators and their graphical counterparts within the Fock space framework.

In the subsequent discussion, we undertake a comparative graphical analysis of H-Toeplitz operators associated with symbols of varying analytic nature. Specifically, we consider cases where the symbol $\phi(z)$ is  purely non-analytic, purely analytic, and a combination of both analytic and non-analytic components.
This classification enables a clearer understanding of how the analytic structure of the symbol influences the graphical behavior of the corresponding operator.

\begin{definition}A directed H-Toeplitz graph \( W \) is defined as $W_n \langle x_1, x_2, \ldots, x_p; y_1, y_2, y_3, \ldots, y_q \rangle$
with vertex set \( \{1, 2, 3, \ldots, n\} \) is a digraph whose adjacency matrix is an H-Toeplitz matrix in which the arc \( (i, j) \) occurs if and only if 
\[
j = (2i - 1) + x_k \quad \text{or} \quad i = \frac{j + 1}{2} + y_l
\]
for some \( 1 \leq k \leq p \) and \( 1 \leq l \leq q \).
\end{definition}
The main diagonal of an H-Toeplitz adjacency matrix of order \( n \times n \) is labelled 0. The \( n - 1 \) distinct diagonals above the main diagonal are labelled \( 1, 2, \ldots, n - 1 \), and those below are also labeled \( 1, 2, \ldots, n - 1 \). Let \( x_1, x_2, \ldots, x_p \) be the upper diagonals containing ones and \( y_1, y_2, \ldots, y_q \) be the lower diagonals containing ones, such that $0 < x_1 < x_2 < \ldots < x_p < n$
and $0 < y_1 < y_2 < \ldots < y_q < n$. Then the corresponding H-Toeplitz operator graph will be denoted by $W_n \langle x_1, x_2, \ldots, x_p; y_1, y_2, y_3, \ldots, y_q \rangle$.
.
\begin{example}
%\textbf{Non-analytic}. 
Let $\phi({z})=2\overline{z}^1+3\overline{z}^2+\overline{z}^3$ and take $\alpha=1$ in the definition of $S_\phi$, then the matrix representation of $S_\phi$ with respect to the orthonormal basis $\left\{e_n \right\}^\infty_{n=0}$ of Fock space $F_\alpha^2$ is given by
\[ S_\phi =
\begin{pmatrix}
0 & 0 & 2 & 0 & 3 & 0 & 1 & 0 & 0 & \cdots \\
0 & 0 & 0 & 0 & 2 & 0 & 3 & 0 & 1 & \cdots \\
0 & 0 & 0 & 0 & 0 & 0 & 2 & 0 & 3 & \cdots \\
0 & 0 & 0 & 0 & 0 & 0 & 0 & 0 & 2 & \cdots \\
0 & 0 & 0 & 0 & 0 & 0 & 0 & 0 & 0 & \cdots \\
0 & 0 & 0 & 0 & 0 & 0 & 0 & 0 & 0 & \cdots \\
0 & 0 & 0 & 0 & 0 & 0 & 0 & 0 & 0 & \cdots \\
0 & 0 & 0 & 0 & 0 & 0 & 0 & 0 & 0 & \cdots \\
\vdots & \vdots & \vdots & \vdots & \vdots & \vdots & \vdots & \vdots & \ddots \\
\end{pmatrix}.
\]
The indicator binary matrix is given by 
\[ S_\phi =
\begin{pmatrix}
0 & 0 & 1 & 0 & 1 & 0 & 1 & 0 & 0 & \cdots \\
0 & 0 & 0 & 0 & 1 & 0 & 1 & 0 & 1 & \cdots \\
0 & 0 & 0 & 0 & 0 & 0 & 1 & 0 & 1 & \cdots \\
0 & 0 & 0 & 0 & 0 & 0 & 0 & 0 & 1 & \cdots \\
0 & 0 & 0 & 0 & 0 & 0 & 0 & 0 & 0 & \cdots \\
0 & 0 & 0 & 0 & 0 & 0 & 0 & 0 & 0 & \cdots \\
0 & 0 & 0 & 0 & 0 & 0 & 0 & 0 & 0 & \cdots \\
0 & 0 & 0 & 0 & 0 & 0 & 0 & 0 & 0 & \cdots \\
\vdots & \vdots & \vdots & \vdots & \vdots & \vdots & \vdots & \vdots & \ddots \\
\end{pmatrix}.
\]
Clearly, we can see that the above matrix is an upper triangular matrix in which all the diagonals are zero.
The corresponding H-Toeplitz graph of the above matrix will be \( W_{\infty}\langle 2,4,6;0\rangle\)  as given below

\begin{figure}[h]
\centering
\includegraphics[width=1\textwidth]{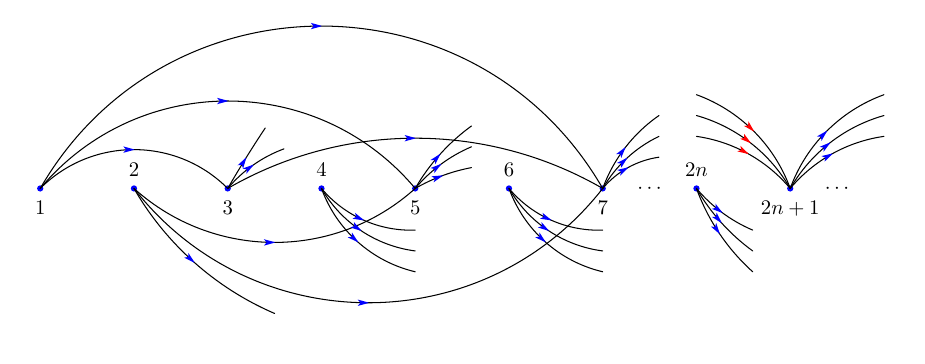}
\caption{\( W_{\infty}\langle 2,4,6;0\rangle\) }
\end{figure}
Consider the infinite digraph \( W_\infty \), where for any vertex \( n \), its out-degree are determined by the relation
\[
j = 2n + x_r - 1, \quad \text{with } x_r \in \{2, 4, 6\}.
\]
This generates the three outgoing edges $j = 2n + 1,  2n + 3, 2n + 5.$
Consequently, every vertex \( n \) has a constant out-degree of \( 3 \), resulting in the out-degree sequence:
\[
\text{Out-degree sequence: } (3, 3, 3, \dots).
\]
The indegree of a vertex \( n \) is derived by inverting the adjacency relation. The predecessors of \( n \) are obtained by solving for \( i \) in
$n = 2i - 1 + x_r, \quad x_r \in \{2, 4, 6\}.$
This yields $i = \frac{n - x_r + 1}{2}.$
For \( i \) to be a valid vertex, \( n - x_r + 1 \) must be a positive even integer. We analyze the indegree on a case-by-case basis
\begin{itemize}%[leftmargin=*]
 \begin{multicols}{5}
 \item[(a)] \(\text{indeg}(1) = 0\).
    \item[(b)] \(\text{indeg}(2) = 0\).
    \item [(c)]\(\text{indeg}(3) = 1\).
    \item [(d)]\(\text{indeg}(4) = 0\).
    \item[(e)] \(\text{indeg}(5) = 2\).   
   
\end{multicols}
\end{itemize}
For \( n \geq 6 \), the indegree depends on the parity of \( n \)
\[
\text{indeg}(n) = 
\begin{cases}
0 & \text{if } n \text{ is even}, \\
3 & \text{if } n \text{ is odd}.
\end{cases}
\]
Thus, the indegree sequence is
$(0, 0, 1, 0, 2, 0, 3, 0, 3, \dots).$ We can see there is no loop in the above digraph.
\end{example}
\begin{example} Let $\phi({z})= 5z + 9z^2 + z^4$ and take $\alpha=1$ in the definition of $S_\phi$, then the matrix representation of $S_\phi$ with respect to the orthonormal basis $\left\{e_n \right\}^\infty_{n=0}$ of Fock space $F_\alpha^2$ is given by
\[ S_\phi =
\begin{pmatrix}
0 & 5 & 0 & 9 & 0 & 0 & 0 & 1 & 0 & \cdots \\
5 & 9 & 0 & 0 & 0 & 1 & 0 & 0 & 0 & \cdots \\
9 & 0 & 5 & 1 & 0 & 0 & 0 & 0 & 0 & \cdots \\
0 & 1 & 9 & 0 & 5 & 0 & 0 & 0 & 0 & \cdots \\
1 & 0 & 0 & 0 & 9 & 0 & 5 & 0 & 0 & \cdots \\
0 & 0 & 1 & 0 & 0 & 0 & 9 & 0 & 5 & \cdots \\
0 & 0 & 0 & 0 & 1 & 0 & 0 & 0 & 9 & \cdots \\
0 & 0 & 0 & 0 & 0 & 0 & 1 & 0 & 0 & \cdots \\
\vdots & \vdots & \vdots & \vdots & \vdots & \vdots & \vdots & \vdots & \ddots \\
\end{pmatrix}.
\]
The indicator binary matrix is given by 
\[ S_\phi =
\begin{pmatrix}
0 & 1 & 0 & 1 & 0 & 0 & 0 & 1 & 0 & \cdots \\
1 & 1 & 0 & 0 & 0 & 1 & 0 & 0 & 0 & \cdots \\
1 & 0 & 1 & 1 & 0 & 0 & 0 & 0 & 0 & \cdots \\
0 & 1 & 1 & 0 & 1 & 0 & 0 & 0 & 0 & \cdots \\
1 & 0 & 0 & 0 & 1 & 0 & 1 & 0 & 0 & \cdots \\
0 & 0 & 1 & 0 & 0 & 0 & 1 & 0 & 1 & \cdots \\
0 & 0 & 0 & 0 & 1 & 0 & 0 & 0 & 1 & \cdots \\
0 & 0 & 0 & 0 & 0 & 0 & 1 & 0 & 0 & \cdots \\
\vdots & \vdots & \vdots & \vdots & \vdots & \vdots & \vdots & \vdots & \ddots \\
\end{pmatrix}.
\]
Clearly, we can see that the above matrix is not an upper triangular matrix.% in which all the diagonals are zero.
The corresponding H-Toeplitz graph of the above matrix will be \( W_{\infty}\langle 1,3,7;1,2,4\rangle\)  as given below.
\begin{figure}[!ht]
\centering
\includegraphics[width=1\textwidth]{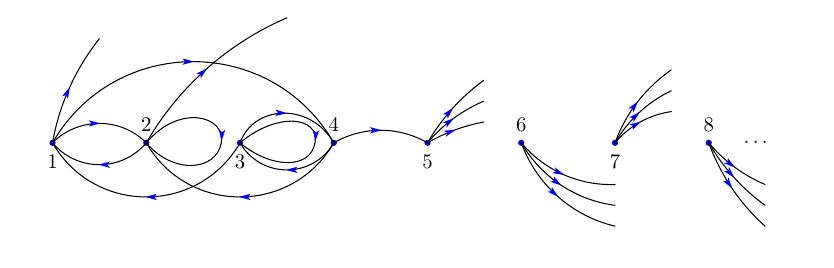}
\caption{\( W_{\infty}\langle 1,3,7;1,2,4\rangle\) }
\end{figure}

From the above figure it is clear that for any vertex $n$, the outdeg(n)=3.\\
Consequently, every vertex \( n \) has a constant outdegree of \( 3 \), resulting in the outdegree sequence:
\[
\text{Out-degree sequence: } (3, 3, 3, \dots).
\]
 We analyze the indegree on a case-by-case basis
\begin{itemize}%[leftmargin=*]
\begin{multicols}{5}
    \item[(a)] \(\text{indeg}(1) = 3\) 
    \item[(b)] \(\text{indeg}(2) = 3\) 
    \item [(c)]\(\text{indeg}(3) = 2\) 
    \item [(d)]\(\text{indeg}(4) = 2\) 
    \item[(e)] \(\text{indeg}(5) = 3\) 
    \item[(f)] \(\text{indeg}(6) = 1\) 
    \item[(g)] \(\text{indeg}(7) = 3\) 
    \item[(h)] \(\text{indeg}(8) = 1\) 
    \end{multicols}
\end{itemize}
For \( n \geq 9 \), the indegree depends on the parity of \( n \)
\[
\text{indeg}(n) = 
\begin{cases}
0 & \text{if } n \text{ is even}, \\
3 & \text{if } n \text{ is odd}.
\end{cases}
\]
Thus, the indegree sequence is
$(3, 3, 2, 2, 3, 1, 3, 1, 3, \dots).$ We can see that there is a loop present atleast the number of terms of analytic symbols in the above digraph.

\end{example}
\begin{example}
% \textbf {Both analytic and non-analytic}. 
 Let $\phi({z})=4z + z^3 + \overline{z}^2+7\overline{z}^3 $ and take $\alpha=1$ in the definition of $S_\phi$, then the matrix representation of $S_\phi$ with respect to the orthonormal basis $\left\{e_n \right\}^\infty_{n=0}$ of Fock space $F_\alpha^2$ is given by 
 
\[ 
S_\phi =
\begin{pmatrix}
0 & 4 & 0 & 0 & 1 & 1 & 7 & 0 & 0 & \cdots \\
4 & 0 & 0 & 1 & 0 & 0 & 1 & 0 & 7 & \cdots \\
0 & 1 & 4 & 0 & 0 & 0 & 0 & 0 & 1 & \cdots \\
1 & 0 & 0 & 0 & 4 & 0 & 0 & 0 & 0 & \cdots \\
0 & 0 & 1 & 0 & 0 & 0 & 4 & 0 & 0 & \cdots \\
0 & 0 & 0 & 0 & 1 & 0 & 0 & 0 & 4 & \cdots \\
0 & 0 & 0 & 0 & 0 & 0 & 1 & 0 & 0 & \cdots \\
0 & 0 & 0 & 0 & 0 & 0 & 0 & 0 & 1 & \cdots \\
\vdots & \vdots & \vdots & \vdots & \vdots & \vdots & \vdots & \vdots & \ddots \\
\end{pmatrix}.
\]
The indicator binary matrix is given by 
\[ 
S_\phi =
\begin{pmatrix}
0 & 1 & 0 & 0 & 1 & 1 & 1 & 0 & 0 & \cdots \\
1 & 0 & 0 & 1 & 0 & 0 & 1 & 0 & 1 & \cdots \\
0 & 1 & 1 & 0 & 0 & 0 & 0 & 0 & 1 & \cdots \\
1 & 0 & 0 & 0 & 1 & 0 & 0 & 0 & 0 & \cdots \\
0 & 0 & 1 & 0 & 0 & 0 & 1 & 0 & 0 & \cdots \\
0 & 0 & 0 & 0 & 1 & 0 & 0 & 0 & 1 & \cdots \\
0 & 0 & 0 & 0 & 0 & 0 & 1 & 0 & 0 & \cdots \\
0 & 0 & 0 & 0 & 0 & 0 & 0 & 0 & 1 & \cdots \\
\vdots & \vdots & \vdots & \vdots & \vdots & \vdots & \vdots & \vdots & \ddots \\
\end{pmatrix}.
\]
Clearly, we can see that the above matrix is not an upper triangular matrix.% in which all the diagonals are zero.
The corresponding H-Toeplitz graph of the above matrix will be \( W_{\infty}\langle 1,4,5,6;1,3\rangle\)  as given below

\begin{figure}[!htbp]
\centering
\includegraphics[width=1\textwidth]{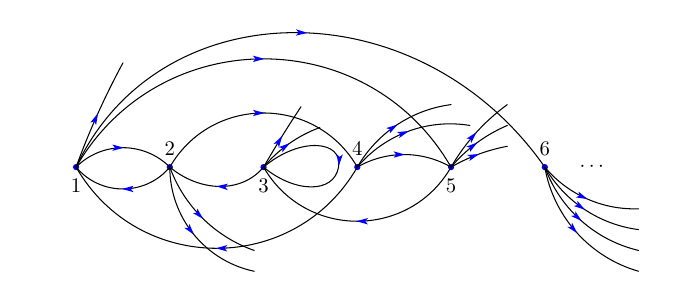}
\caption{\( W_{\infty}\langle 1,4,5,6;1,3\rangle\) }
\end{figure}

From the above figure it is clear that for any vertex $n$, the outdeg(n)=4 .
\end{example}
\subsection{ Comparative Graph Analysis}

The graphical representations of H-Toeplitz operators on the Fock space provide a compelling means of analyzing the underlying structure of these operators through directed adjacency matrices. This study considered three distinct classes of symbols: purely analytic, purely non-analytic, and a mixture of both, each yielding characteristic patterns in the associated graphs.

\begin{itemize}
 \item[(a)] \textbf{Purely non-analytic symbols} such as those involving only anti-analytic terms like $\overline{z}, \overline{z}^2, \ldots$ etc. yield \emph{upper triangular} adjacency matrices in their associated H-Toeplitz graphs. A key structural feature of these matrices is that the main diagonal is identically zero, reflecting the absence of self-loops in the graph. 

 \item[(b)] \textbf{Purely analytic symbols}  comprising terms such as $z, z^2, z^3, \ldots$ etc. yield H-Toeplitz adjacency matrices that are generally \emph{non-triangular}. Unlike the strictly upper triangular structure seen in the non-analytic case, the presence of analytic monomials introduces non-zero entries along and below the main diagonal. As illustrated by the graph $W_\infty\langle 2,4,6; 0\rangle$, each vertex admits a fixed number of outgoing edges, dictated by the degrees of the analytic terms. Notably, such graphs may contain self-loops, and the number of loops is at least equal to the number of analytic terms in the symbol.

   \item[(c)] \textbf{Mixed analytic and non-analytic symbols} involving both analytic terms (such as $z, z^2, \ldots$ etc.) and anti-analytic terms (such as $\overline{z}, \overline{z}^2, \ldots$  etc.) produce H-Toeplitz adjacency matrices that are generally \emph{non-triangular}. These matrices exhibit non-zero entries both above and below the main diagonal, reflecting the dual nature of the symbol. The corresponding graphs demonstrate a blend of forward and backward edge flows, often leading to complex patterns with overlapping arcs and enhanced connectivity. As observed in examples such as $W_\infty\langle 1,4,5,6; 1,3\rangle$, these graphs feature multiple self-loops and dense adjacency relations, with both the indegree and outdegree sequences showing irregular but interpretable behavior.
\end{itemize}

From the comparative perspective, the out-degree sequences remain uniformly bounded and often constant across all cases due to the fixed number of terms in the symbol. However, the indegree sequences and loop structures vary significantly and serve as key indicators of the symbol's analytic makeup. These findings highlight the power of graph-theoretic tools in understanding operator-theoretic phenomena in the Fock space and open new avenues for analyzing more generalized or weighted symbol cases in future research.

\section{Conclusion}\label{sec8}
This study establishes a comprehensive framework for analyzing H-Toeplitz operators on the Fock space $F_{\alpha}^{2}$. We have derived their explicit matrix representations, which reveal a hybrid structure intertwining features of classical Toeplitz and Hankel operators. Our analysis of their algebraic properties culminated in a precise characterization of commutativity for operators with harmonic symbols and a definitive proof that no non-zero H-Toeplitz operator can be a Hilbert-Schmidt operator, underscoring their inherent unboundedness in this norm.A major contribution lies in the spectral analysis of these operators. By leveraging the Mellin transform, we obtained complete characterizations of hyponormality and normality for H-Toeplitz operators with quasi-homogeneous symbols. We established that hyponormality is equivalent to a specific inequality between Mellin transforms, while normality requires a precise equality condition that holds automatically for radial symbols. A central and novel contribution of this work is the introduction of directed H-Toeplitz graphs. This innovative approach translates the matrix data of an operator into a directed graph, allowing us to visualize and classify operators based on their symbols.

In conclusion, our findings significantly advance the theoretical understanding of H-Toeplitz operators by resolving fundamental questions about their commutativity, compactness, and spectral properties (hyponormality and normality), while also pioneering a new dialogue between operator theory on analytic function spaces and graph theory. This bridge opens up promising avenues for future research, including the analysis of weighted symbols, the spectral properties of these operators, and further exploration of the rich combinatorial patterns embedded within their graphical representations.

\section{Declarations}\label{sec9}

The third author gratefully acknowledges support from the UGC Research Grant No. F.82-44/2020(SA-III), Sr. No. 201610157825, as well as from the CSIR Grant No. 09/0476(15236)/2022-EMR-I. The fourth author is supported by CSIR Grant No. 09/0476(15237)/2022-EMR-I. The authors declare that they have no conflict of interest or competing interests, financial or non-financial, with respect to the subject matter or materials discussed in this manuscript. All authors have contributed equally to this work.

\end{document}